\documentclass[12pt]{amsart}
\usepackage[utf8]{inputenc}
\usepackage{graphicx}
\usepackage{amsfonts}
\usepackage{amsmath,amssymb,amsthm}
\usepackage{tabularray}
\usepackage[numbers]{natbib}
\usepackage{url}
\usepackage[pagewise]{lineno} %\linenumbers
\newtheorem{theorem}{Theorem}
\newtheorem{lemma}{Lemma}
\newtheorem{question}{Question}

\title{Markov Number Graphs Extended to all Integer Triples}
\author{Spencer Scutt, Mark Turpin}
\date{May 2025}

\begin{document}

\begin{abstract}
    We study the graphs generated when the formula for linking Markov triples is applied to general triples of integers.  
    We find there are a finite number of equivalence classes of graphs, each with particular properties.
\end{abstract}

\maketitle

\author{Spencer Scutt}
\address{Enfield, CT 06082}
\email{smartypantsspencer@gmail.com}

\author{Mark Turpin}
\address{Department of Mathematics, University of Hartford, West Hartford, CT 06117}
\email{mturpin@hartford.edu}



\section{Introduction}

Markov numbers are the solutions to the Diophantine equation
$a^2 + b^2 + c^2 = 3abc$. Markov numbers can be considered individually, or as part of a triple of integers.
For example, The first three Markov triples, in ascending order, are $ (1,1,1), (1,1,2), (1,2,5) $, and the first three Markov numbers are $ 1, 2, 5 $.  
For a comprehensive summary of the Markov numbers, see \cite{Aigner}.  
For connections of Markov numbers to Diophantine approximation, see \cite{Cassels}. 
For a generalization of Markov numbers see \cite{Karpenkov}.

Given any Markov triple, one can take any entry of the triple and produce another triple, which we refer to as a {\it linked} triple.  
We display this as follows:
$$ (a,b,c) \leftrightarrow (a',b,c) \; \mbox{where} \; a' = 3bc-a. $$
The triples we work with are unordered, so $ (a,b,c) = (b,a,c) $, for example.

Given $(a, b, c)$, we may find 3 more linked triples: $(3bc-a, b, c)$, $(a, 3ac-b, c)$, and $(a, b, 3ab-c)$. 
This process of {\it Vieta jumping}, also called {\it root flipping}, allows one to start with any Markov triple and generate all others (of which there are countably many).
The relationships between triples may be organized into a graph: \newline
\includegraphics[width=\columnwidth]{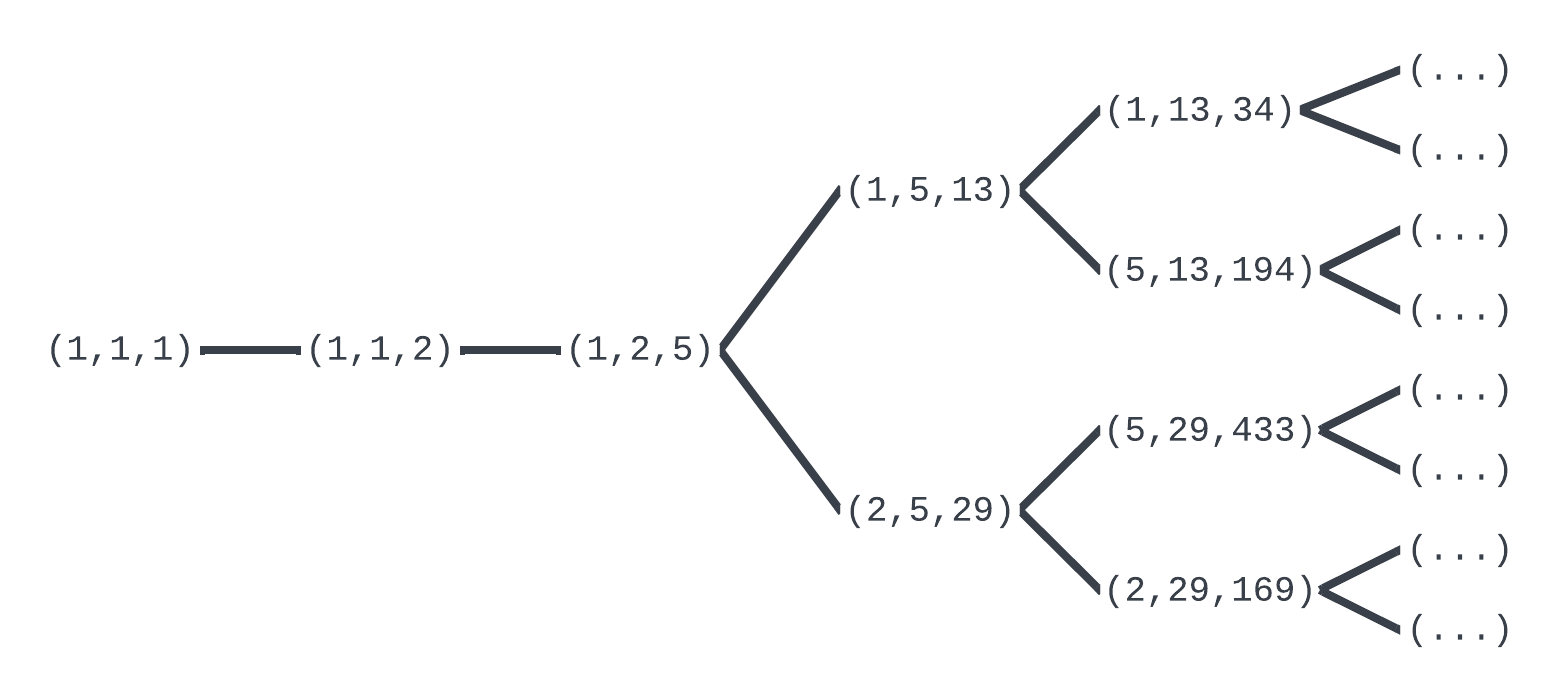}

The vertices of this graph are the Markov triples, and the edges of the graph are the Vieta jumpings.  
We note the Vieta jumpings are invertible, and the directed edges are suppressed.  
For more details on Markov triples and the associated graph, see \cite{Wikipedia}.

\subsection{Arbitrary Application of the Method}
This paper explores the graphs created by the Vieta jumping procedure when starting with arbitrary triples of integers, both positive and negative, not necessarily Markov triples.  For example:

\begin{center}
\includegraphics[scale=.18]{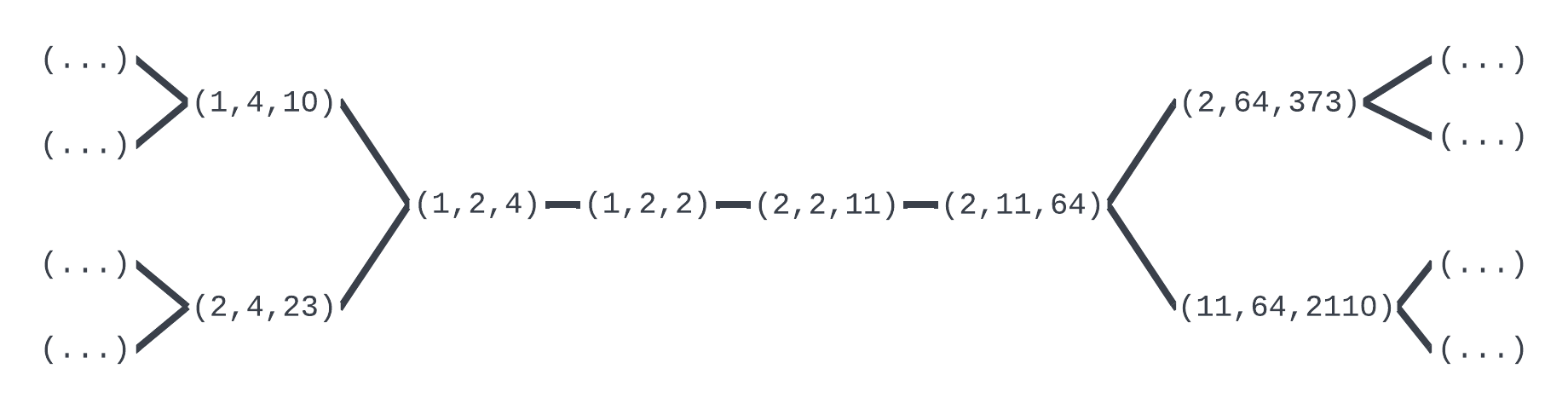}
\end{center}

This graph could have been generated by any of the triples included among the vertices, for example (1, 2, 4).
Note that this graph is not isomorphic to the graph created by the Markov triples (there exists no vertex of degree 1).
We will show there exist exactly nine distinct equivalence classes of graphs, when generated by arbitrary triples of integers.

\section{Fundamental Ideas}
A few fundamental lemmas and terminology must be laid out before classification of graphs can begin.
Note the triples are not ordered. For example, (1,2,4) and (2,4,1) are the same element in the space of triples, and therefore the same vertex in the graph.

\subsection{Seeds}
By starting with an arbitrary triple of integers and applying the Vieta jumping, we generate a graph.  
The arbitrary starting triple is referred to as a {\it seed}.
Any of the triples within this graph could be the seed of the graph, if it had been selected as a starting triple.
By the invertibility of the Vieta jumping, each triple belongs to exactly one graph.
If two graphs, generated by different seeds, are isomorphic, we consider them equivalent.
By taking the entire collection of graphs generated by all possible triples, we can group the resulting graphs into these equivalence classes.  
We show that there exist exactly nine equivalence classes.

\subsection{Norm of a Vertex}
Given vertex $V = (a,b,c)$, we define the {\it norm} of $V$:
$$ ||V|| = |a| + |b| + |c|. $$

\subsection{Bases}
Given a graph and a vertex $ V $ contained in the graph, we say $ V $ is the {\it base} of the graph if and only if $ || V || \le || W || $ for all vertices $ W $ contained in the graph.
This is to say, the base of a graph is the vertex with the minimum norm, which is guaranteed to exist in any set of integer triples.
We note that while every graph has a base, the base might not be unique. \newline
When we show visuals of our graphs, bases are marked in green. However, note that these graphs as mathematical objects are not colored. The green is merely a visual aid for the reader.

\subsection{Graph equivalence}
For our definition of graph equivalence, we use the notion of graph isomorphism of undirected, unlabeled, non-weighted graphs. \cite{Wiki-isomorphism}

\subsection{Infinite Binary Tree}
Many of the graphs we examine will have infinite subgraphs that are locally 3-regular \cite{Wiki-tree} except for one vertex of degree 2.
We call such a subgraph an {\it infinite binary tree}.  
We call the exceptional vertex the {\it root}.

\begin{center}
\includegraphics[scale=.7]{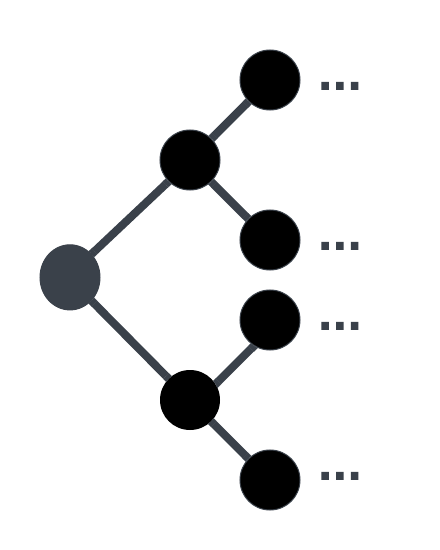}
\end{center}
Here, the leftmost vertex is the root.

\subsection{The Main Theorem}

\begin{theorem}
There exist exactly nine equivalence classes of graphs formed by the Markov Vieta jumpings applied to unordered triples of integers.
\end{theorem}

The following table lists the nine equivalence classes. 
The diagrams on the right fully show the structure of each graph, recalling that a line ending in ellipses denotes the root of an infinite binary tree. The triplets on the left demonstrate the best-known mathematical generalization of bases belonging to each class.
Recall that bases are given in green as a visual aid, and this visual aid is not mathematically significant.

\begin{table}[ht]
    \centering 
    \resizebox{\textwidth}{!}{
    \begin{tabular}{|c|c|c|}
    \hline
        Class & Base(s) & Graph \\
    \hline \hline
        1 & $ (0,0,0) $ & \includegraphics[scale=0.14]{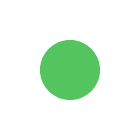} \\
    \hline
        2 & $ (0,0,a), $ $ (0,0,-a) $ & \includegraphics[scale=0.18]{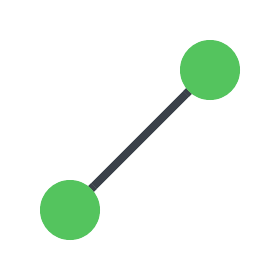} \\
    \hline
        3 & $ (0,a,b), $  $ (0,-a,b), $  $ (0,-a,b), $  $ (0,-a,-b) $ & \includegraphics[scale=0.14]{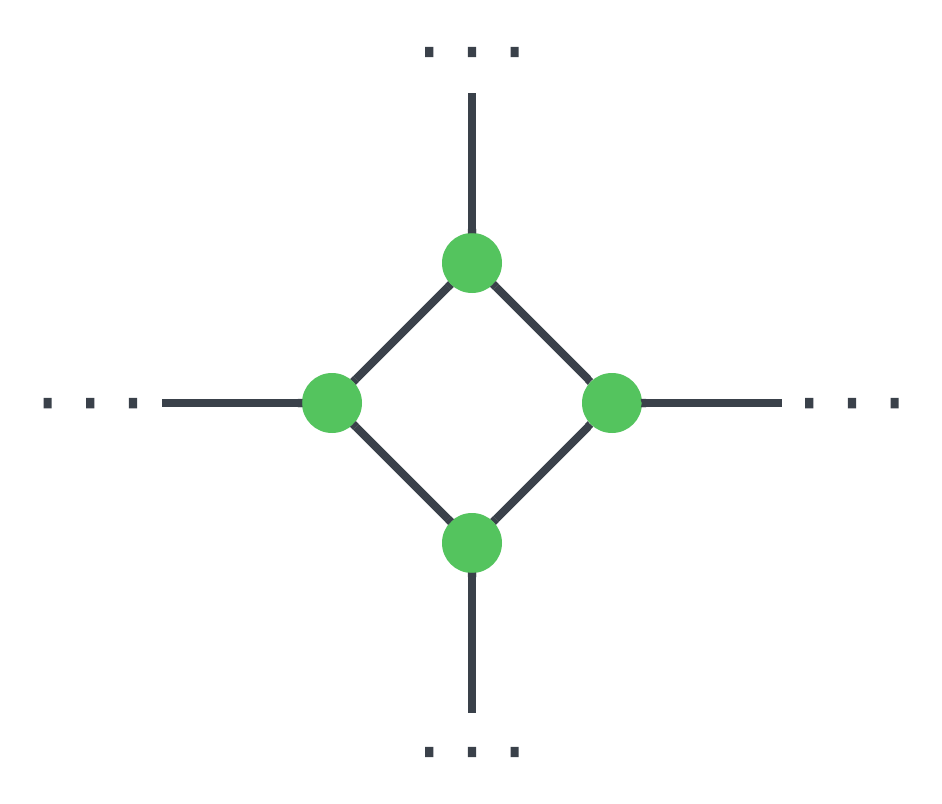} \\
    \hline
        4 & $ (0,a,a), $ $ (0,a,-a), $ $ (0,-a,-a) $  & \includegraphics[scale=0.18]{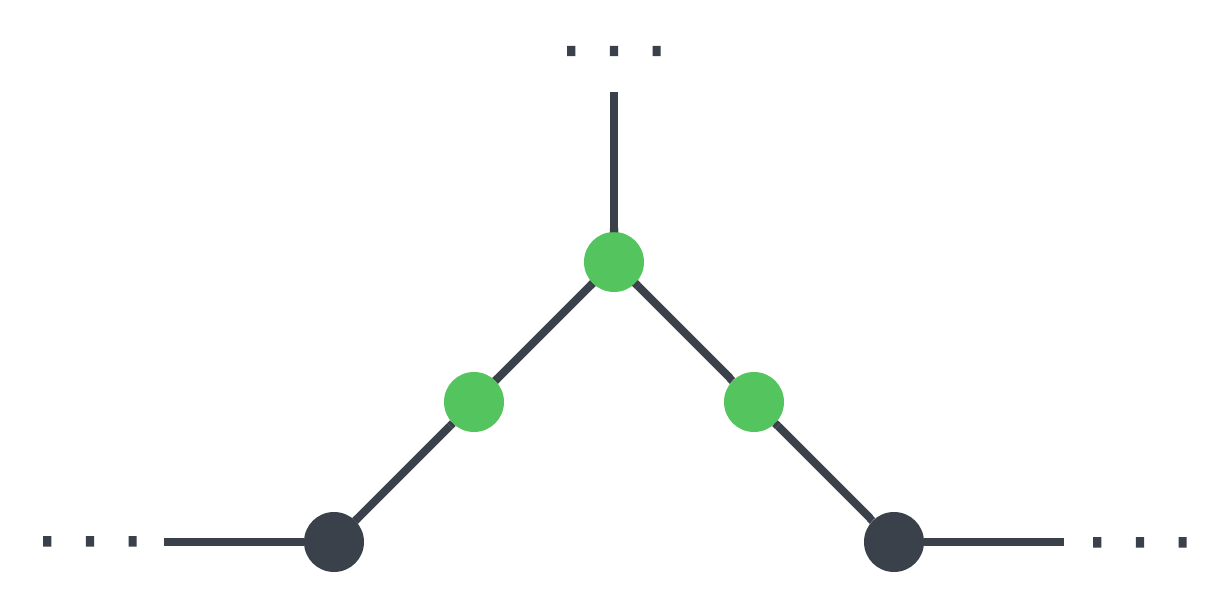} \\
    \hline
        5 & $ (a,a,a) $ & \includegraphics[scale=0.18]{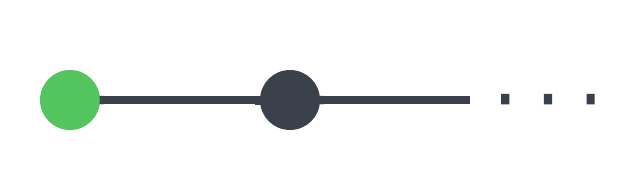} \\
    \hline
        6 & $ (a,a,b) $ & \includegraphics[scale=0.18]{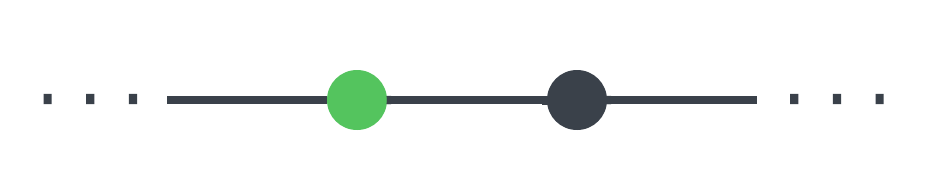} \\
    \hline
        7 & $ (a,a,3a^2/2) $ & \includegraphics[scale=0.18]{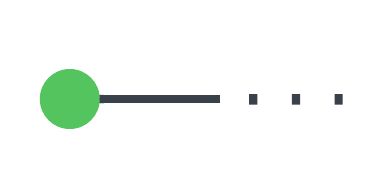} \\
    \hline
        8 & $ (n,2m,3nm) $ where $n,m$ non-zero  & \includegraphics[scale=0.18]{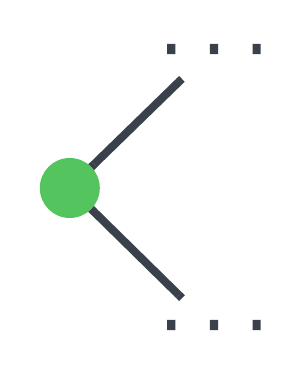} \\
    \hline
        9 & $(a,b,c)$ & \includegraphics[scale=0.18]{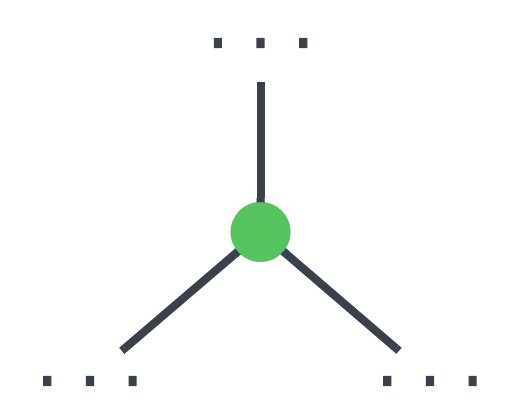} \\
    \hline     
    \end{tabular} } 
\end{table}

In the table, $ a, b, c $ are mutually distinct, non-zero integers. Again, note that the color of vertices are visual aids and not mathematically significant.
\newline * Not every triple of these forms are bases. However, every base takes one of these forms.

\subsection{Two important lemmas}

We begin with two lemmas which we use throughout our classification.

\begin{lemma}
\label{binary}
Given $ (a,b,c) $ with $ 0 < |a| \leq |b| < |c| $, the triples $ (a',b,c) $ and $ (a,b',c) $ satisfy $ |b| < |c| < |a'| $ and $ |a| < |c| < |b'| $ where $ a' = 3bc - a $ and $ b' = 3ac-b $.
\end{lemma}
\begin{proof}
$ |3bc-a| \ge 3|b||c|-|a| > 3|b||c|-|c| = |c|(3|b|-1) > |c| $. 
We note: $ 3 |b| - 1 > 1 \Leftrightarrow |b| > 2/3 $, which is always true. 
A similar proof works for $ |b'| > |c| $.
\end{proof}
We interpret this as saying the norm of a linked triple with a greatest element increases, whenever one of the two smaller elements of the triple are replaced.  
Furthermore, starting with $ (a,b,c) $ and having $ |a'| > |c| $ and $ |b'| > |c| $ , and noting that $ |b| < |c| < |a'| $ and $ |a| < |c| < |b'| $, allows us to apply the same lemma to the triples $ (b,c,a') $ and $ (a,c,b') $ giving rise to a infinite binary tree, as seen in the Markov graph above, starting with the root triple $ (1,2,5) $.

Note that circuits are not possible in infinite binary trees, as a circuit would imply the existence of a triple $ V= (a,b,c) , |a| < |b| < |c| $ with two linked triples having norm less than $ V $, violating Lemma \ref{binary}.

Also note the triples $ (a,b,c) $ and $ (a,b,c') , (c' = 3ab-c) $ have no such universal relationship between their norms.

We give another lemma which reveals a symmetry present when the signs of two of three numbers in a triple are switched.
\begin{lemma}
\label{symmetry}
Given triple $ (a,b,c) $ with $ b \neq 0, c \neq 0, |a| \neq |b|, |b| \neq |c|, |a| \neq |c| $, and adjacent triples $ \{ (a',b,c), (a,b',c), (a,b,c') \} $, the triples adjacent to $ (a,-b,-c) $ will be 
$$ \{(a',b,c), (a,-b',-c), (a,-b,-c') \}. $$
\end{lemma}
\begin{proof}
The calculations are 
$$ a' = 3bc - a = 3(-b)(-c) - a, $$
$$ b' = 3ac - b = -(3a(-c) - (-b)), $$
$$ c' = 3ab - c = -(3a(-b) - (-c)). $$
\end{proof}
The import of Lemma \ref{symmetry} is the establishment of a symmetry.
When two of the three numbers in a triple have their sign switched, the resulting connecting triples share the same property (2 of their signs are also switched).
So for example, the triple (2,3,4) will be linked to three triples (34,3,4), (2,21,4), (2,3,14).  
The triple with two of three signs altered, (-2,3,-4) will be linked to (-34,3,-4), (-2,21,-4), (-2,3,-14).

\section{Classification}
To begin our classification of the graph equivalence classes, we recall that every graph has a base, possibly not unique.
Each base is a (non-ordered) triple of integers.
Each triple of integers either contains one or more 0's, or it does not.

If a triple contains a 0, it will adopt one of the following four forms:
$$ (0,0,0), (0,0,a), (0,a,b), (0,a,a), $$
where $ a $ and $ b $ are integers each non-zero with $ a \neq b $.

If a triple does not contain a 0, it will adopt one of the following three forms:
$$ (a,a,a), (a,a,b), (a,b,c), $$
where $ a, b, $ and $ c $ are non-zero integers with $ a \neq b, b \neq c, $ and $ a \neq c $.

We will work through the seven possible forms for the base of a graph, and for each form we will identify the graph isomorphism class or classes that follow from the chosen base.
For each isomorphism class, we begin with an abstract diagram representing the general form of the graph. These are very similar to those shown in the table on page 4.
We re-explain explicitly the meaning of these diagrams:

Edges drawn between vertices represent a Vieta jumping from one triple to another. Ellipses denote, as on page 4, the root of an infinite binary tree. Green vertices are a visual aid to identify bases to the reader.

We also draw loop edges to represent a Vieta jumping from a triple to itself. 
These edges are provided only for the reader's intrigue, as this proof does not mention them. 
Indeed, we claim that Theorem 1 remains true whether loop edges are included or excluded from the graphs under examination.

Following this abstract diagram, we show an example of a graph composed of triplets that belongs to the class.

\subsection{Base $(0,0,0)$}
When the base of a graph is $ (0,0,0) $, this vertex is the only vertex in the entire graph, as the Vieta jumping links this vertex to itself. 
Calculating $ 3\cdot 0 \cdot 0 - 0 = 0 $ leads to
$$ (0,0,0) \leftrightarrow (0,0,0). $$
\begin{center}
\includegraphics[scale=0.35]{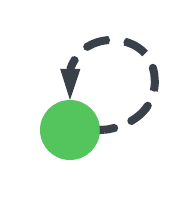}
\end{center}
We call this graph class {\bf Class 1}.
\begin{center}
\includegraphics[scale=0.35]{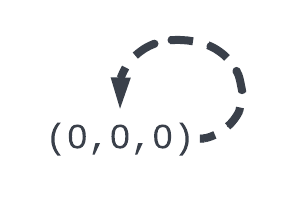}
\end{center}

\subsection{Base $(0,0,a)$ }
Here $ a $ is any non-zero integer.
Calculating $ 3 \cdot 0 \cdot a - 0 = 0 $ and $ 3 \cdot 0 \cdot 0 - a = -a $ leads us to
$$ (0,0,a) \leftrightarrow (0,0,a) \leftrightarrow (0,0,-a) \leftrightarrow (0,0,-a). $$
We call this graph {\bf Class 2}.
\begin{center}
\includegraphics[scale=0.35]{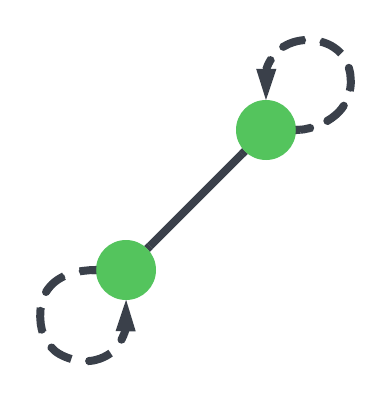}
\end{center}
Class 2 graphs have exactly two bases, namely $ (0,0,a) $ and $ (0,0,-a) $.
\begin{center}
\includegraphics[scale=0.35]{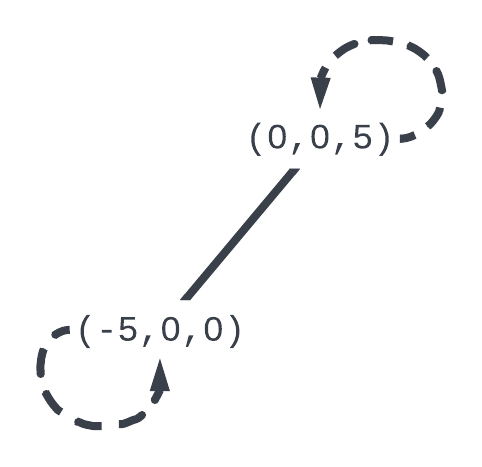}\newline
\end{center}
The graph with bases $(0,0,5)$ and $(-5,0,0)$ is an example.

\subsection{Base $(0,a,b)$ }
Here $ a,b $ are non-zero integers with $|a| \neq |b|$ .
We can calculate the triples linked to our base
$$ (0,a,b) \leftrightarrow (3ab,a,b), $$
$$ (0,a,b) \leftrightarrow (0,-a,b), $$
$$ (0,a,b) \leftrightarrow (0,a,-b). $$
Further calculations reveal four bases, each linked to two other bases:
$$ (0,a,b) \leftrightarrow (0,a,-b) \leftrightarrow (0,-a,-b) \leftrightarrow (0,-a,b) \leftrightarrow (0,a,b). $$

We call this graph {\bf Class 3}.
\begin{center}
\includegraphics[scale=0.2]{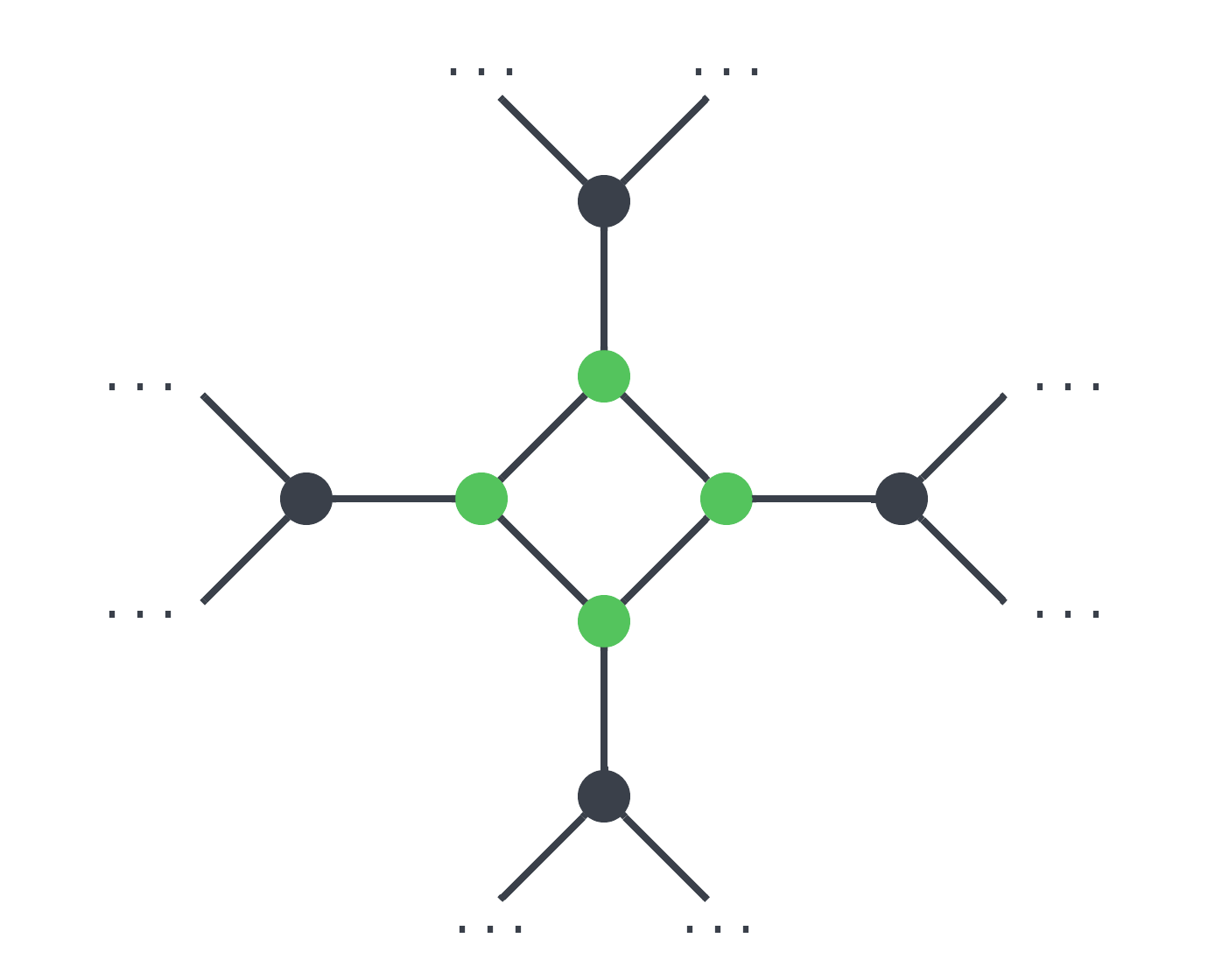}
\end{center}
We establish that the four bases are indeed the vertices with the minimum norms.
Clearly $ ||(0,a,b)|| < ||(a,b,3ab)||$ , but what about the remaining vertices?

It follows from Lemma \ref{binary} that the triples adjacent to the bases, $ (a,b,3ab) $, $ (-a,b,-3ab) $, $ (-a,-b,3ab) $, $ (a,-b,-3ab) $, are the roots of infinite binary trees.

In our Class 3 graph, we have four branches which are themselves not only isomorphic, being infinite binary trees, but also have exactly the same triples, up to the two out of three sign adjustment explained in Lemma \ref{symmetry}.

We note that Class 3 graphs have 4 bases, which will prove to be the most of any class, and these bases form a circuit, also a unique feature.

\includegraphics[scale=0.22]{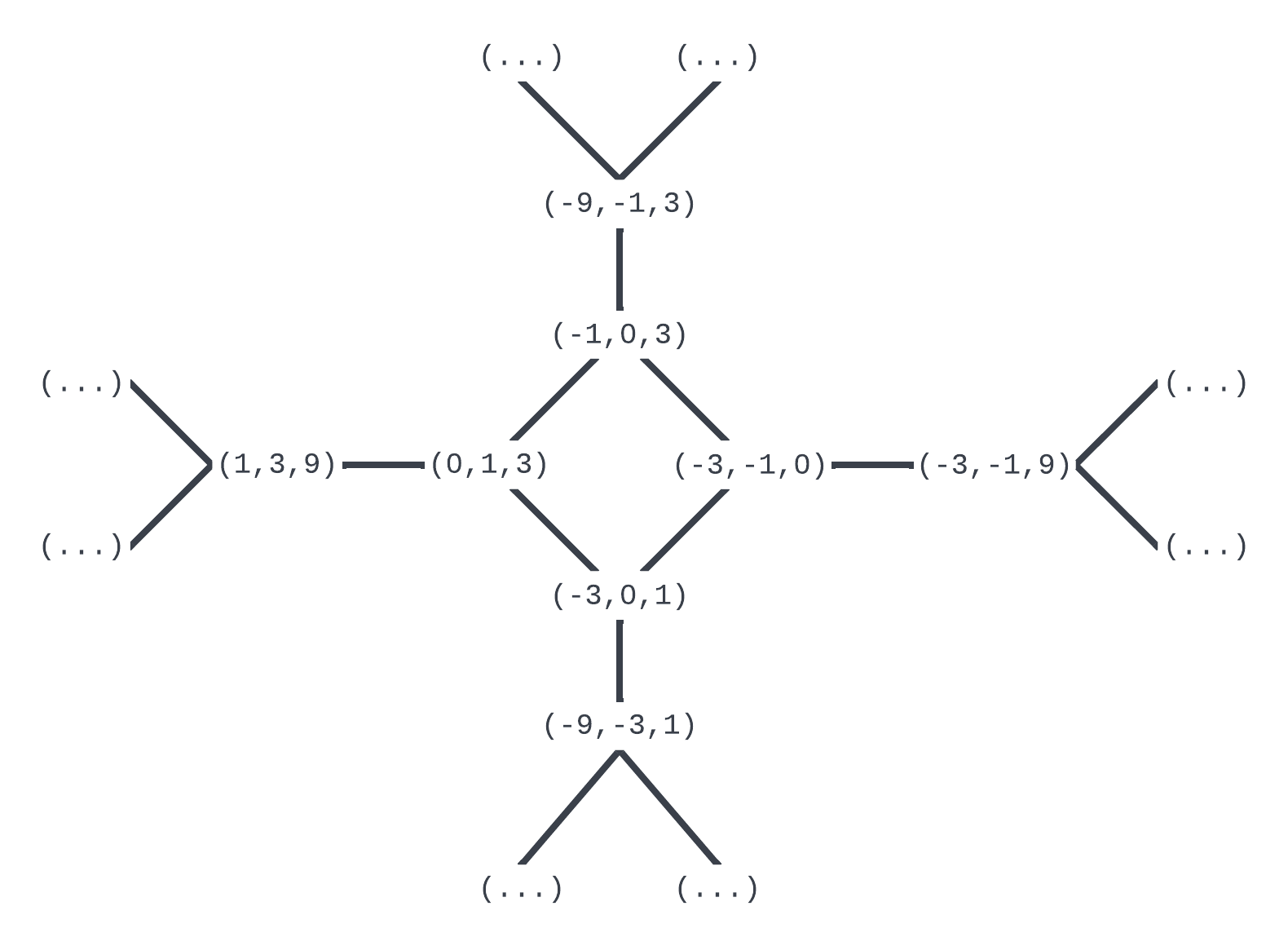}\newline

An example is the graph generated by $(0,1,3)$. These graphs exhibit the symmetry implied by Lemma 2.

\subsection{Base $(0,a,a)$ or $(0,-a,a)$ }
Here $ a \neq 0 $.

Calculating, we can see the following bases are linked:
$$ (0,a,a) \leftrightarrow (0,a,-a) \leftrightarrow (0,-a,-a). $$
We call this graph {\bf Class 4}.

Class 4 graphs could intuitively (but not precisely) be considered a special case of class 3 graphs; when $|a| = |b|$, there are only 3 possible bases instead of 4.

\includegraphics[scale=0.22]{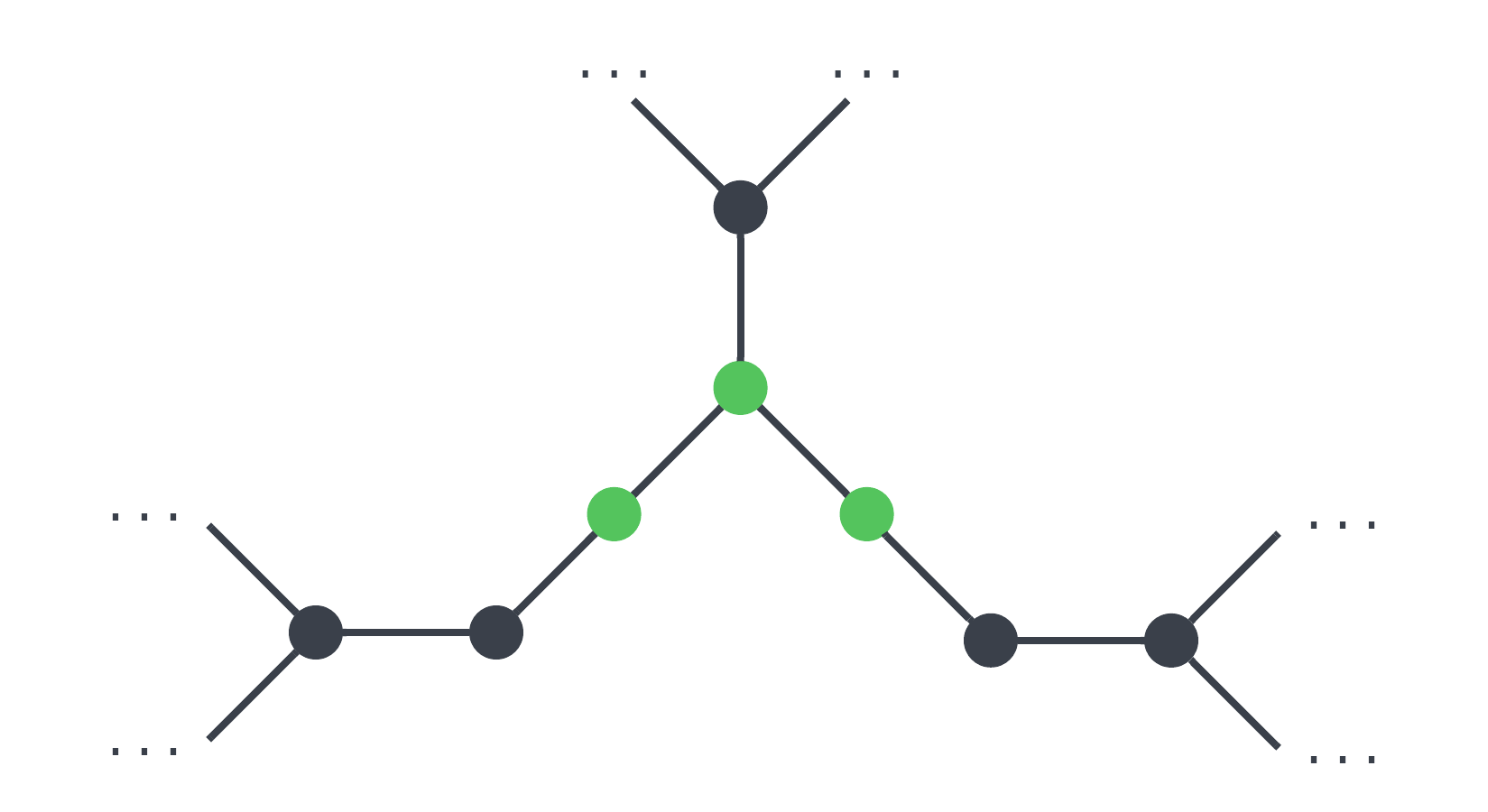}

By Lemma \ref{symmetry} we see the branches of the graph linked to $ (0,a,a) $ and $ (0,-a,-a) $ will be isomorphic. 
Calculating triples linked to $ (0,a,a) $ we find
$$ (0,a,a) \leftrightarrow (3a^2,a,a) \leftrightarrow (3a^2,a,9a^3-a), $$ and similarly
$$ (0,-a,-a) \leftrightarrow (3a^2,-a,-a) \leftrightarrow (3a^2,-a,-(9a^3-a)). $$
Given $ a > 0 $, we establish that for $ a \ge 1 $, we have $ a < 3a^2 < 9a^3 - a $.
By Lemma \ref{symmetry} the branches from $ (0,a,a) $ and $ (0,-a,-a) $ are isomorphic.
From Lemma \ref{binary} we know that $ (3a^2,a,9a^3-a) $ is the root of an infinite binary tree.

For the base $ (0,a,-a) $ we calculate the linked triple
$$ (0,a,-a) \leftrightarrow (-3a^2,a,-a), $$ and by Lemma 1, $(-3a^2,a,-a)$ is the root of an infinite binary tree.

\begin{center}
\includegraphics[scale=0.18]{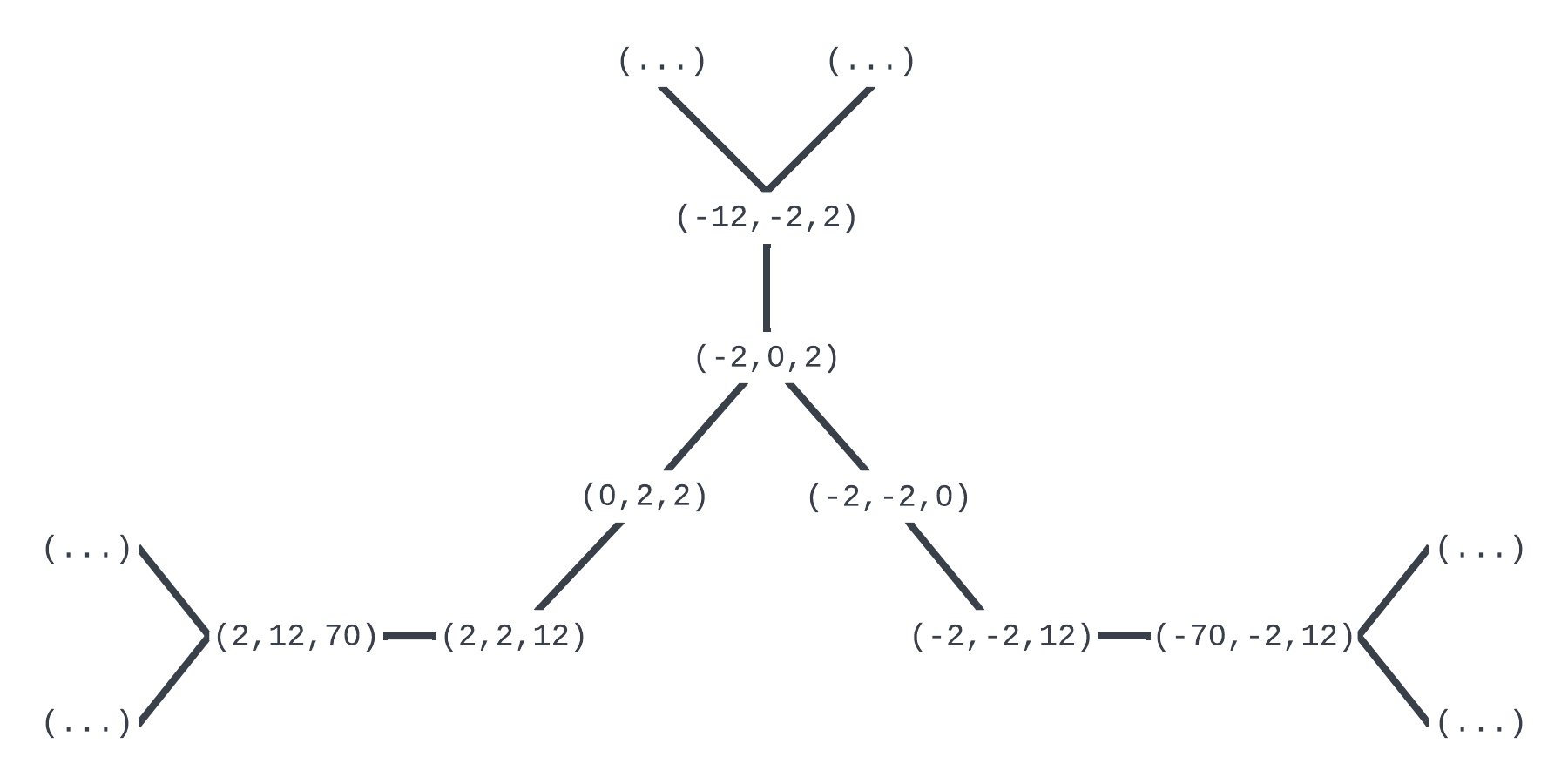}\newline
\end{center}

\subsection{Base $(a,a,a)$ }

Here $ a \neq 0 $ .
Calculating the linked triples we have:
$$ (a,a,a) \leftrightarrow (a,a,3a^2-a) \leftrightarrow (9a^3-3a^2-a,a,3a^2-a). $$ 
Note that for $ a > 0 $,
$$ a < a(3a-1) = 3a^2-a. $$
Also,
$$ 2/3 < a \Rightarrow 6a^2 < 9a^3 \Rightarrow 3a^2 - a < 9a^3 - 3a^2 -a. $$
which means $ (9a^3-3a^2-a,a,3a^2-a) $ is the root of an infinite binary tree by Lemma \ref{binary}.
If $ a < 0 $ then $ (9a^3-3a^2-a,a,3a^2-a) $ is a triple where $ 9a^3-3a^2-a < 0 $, $ a < 0 $, $ 3a^2-a > 0 $, so by Lemma \ref{symmetry} we still have the root of an infinite binary tree.

We call this graph {\bf Class 5}.

\begin{center}
\includegraphics[scale=0.25]{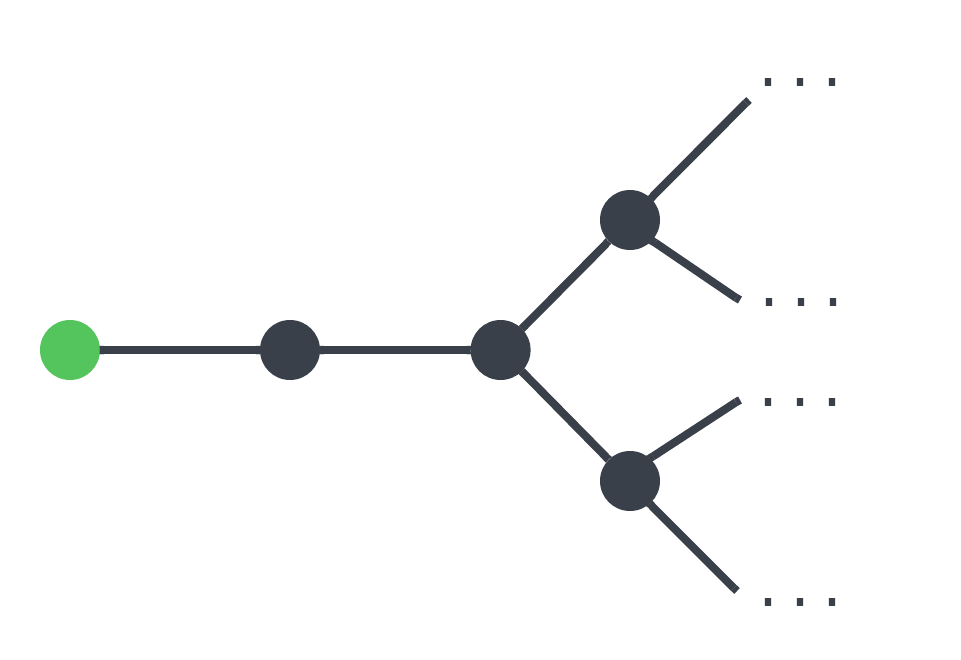}
\end{center}

Classes 5 and beyond contain strictly no zeros. 
Class 5 itself includes the original graph of Markov triples with a base of $ (1,1,1) $.

\includegraphics[width=\columnwidth]{1-1-1}

\subsection{Base $(a,a,b) $}

Here $ a \neq 0, b \neq 0, a \neq b $.
We begin by noting the following linked triples:
$$ (a,a,b) \leftrightarrow (a,a,b') , b' = 3a^2 -b. $$
We observe the following three restrictions:
\begin{itemize}
    \item $ b' \neq 0, (b \neq 3a^2) $ as this graph would have $ (a,a,0) $ as a base and would be a Class 4 graph. 
    \item $ b' \neq a, (b \neq 3a^2 -a) $ as this graph would have $ (a,a,a) $ as a base and would be a Class 5 graph.
    \item $ b' \neq b, (b \neq 3a^2/2) $ this case will be considered later as Class 7.
\end{itemize}

The graph with these restrictions is called {\bf Class 6}.

The relative size of $ |b| $ and $ |b'| $ determine which of the two is the actual base.

In the special case of $ (a,a,-a) $ we have
$$ (-3a^2-a,a,-a) \leftrightarrow (a,a,-a) \leftrightarrow (a,a,3a^2+a) \leftrightarrow (3a(3a^2+a)-a,a,3a^2+a), $$
and each of the end triples is the base of an infinite binary tree from Lemma \ref{binary}.

Next we assume $ |a| \neq |b| $ and look at the linked triples:
$$ (3ab-a,a,b) \leftrightarrow (a,a,b) \leftrightarrow (a,a,b') \leftrightarrow (3ab'-a,a,b'). $$
We see $ |3ab-a| = |a(3b-1)| > |a| $ and $ |a(3b-1)| > |b| $, which allows us to apply Lemma \ref{binary} and determine that $ (3ab-a,a,b) $ is the root of a binary tree.
The same argument applies to the triple $ (3ab'-a,a,b') $ linked to $ (a,a,b') $.

We note the signs of $ a $ and $ b $ do not affect the results.
If $ a < 0 $ and $ b > 0 $ then $ (a(3b-1),a,b) $ will have the sign pattern $ (-,-,+) $, and by Lemma \ref{symmetry} will follow the graph as $ (-a(3b-1),-a,b) $.
If $ a > 0 $ and $ b < 0 $ then $ (a(3b-1),a,b) $ will have the sign pattern $ (-,+,-) $.
If $ a < 0 $ and $ b < 0 $ then $ (a(3b-1),a,b) $ will have the sign pattern $ (+,-,-) $.
In the latter two cases, the same argument applies.

\begin{center}
\includegraphics[scale=0.22]{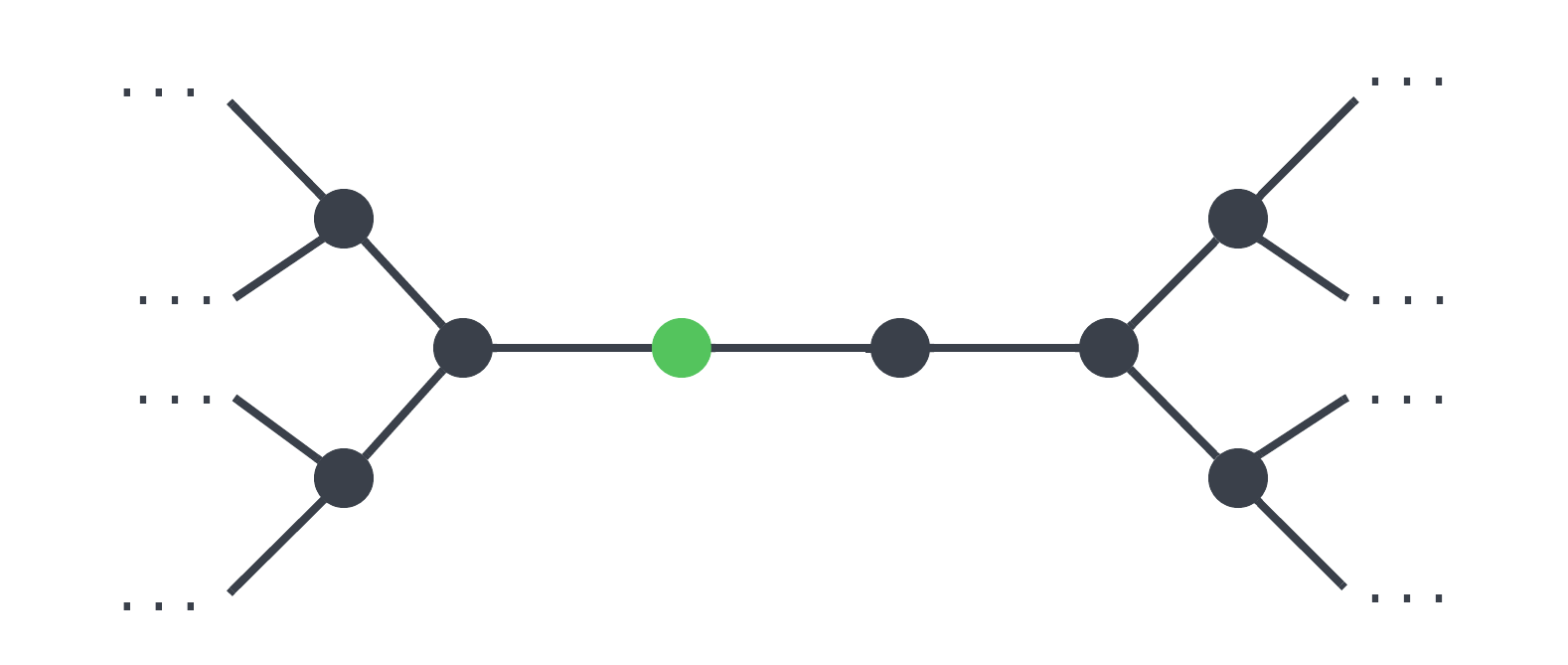}
\end{center}

\begin{center}
\includegraphics[width=\columnwidth]{1-2-2}
\end{center}

The graph with base $(1,2,2)$, shown at the start of the paper, is an example.

\subsection{Base $ (a,a,3a^2/2) $ }

This is the case deferred from the previous section.
From the equation $ b= 3a^2/2 $ we have $ 2b = 3a^2 $ and know $ a $ will be even, and therefore so will $ b $.  
Calculating $ a' = 3ab - a = a(3b-1) $ and $ b' = 3a^2-3a^2/2 = 3a^2/2 = b $, we have the linked triples:
$$ (a,a,b) \leftrightarrow (a,a,b) \leftrightarrow (a(3b-1),a,b). $$
The base is linked to itself, and $ (a,b,a(3b-1)) $ satisfies the conditions of Lemma \ref{binary} and $ |a| < |b| < |a(3b-1)| $ which means it is the root of a infinite binary tree.

We call this graph {\bf Class 7}.

\begin{center}
\includegraphics[scale=0.25]{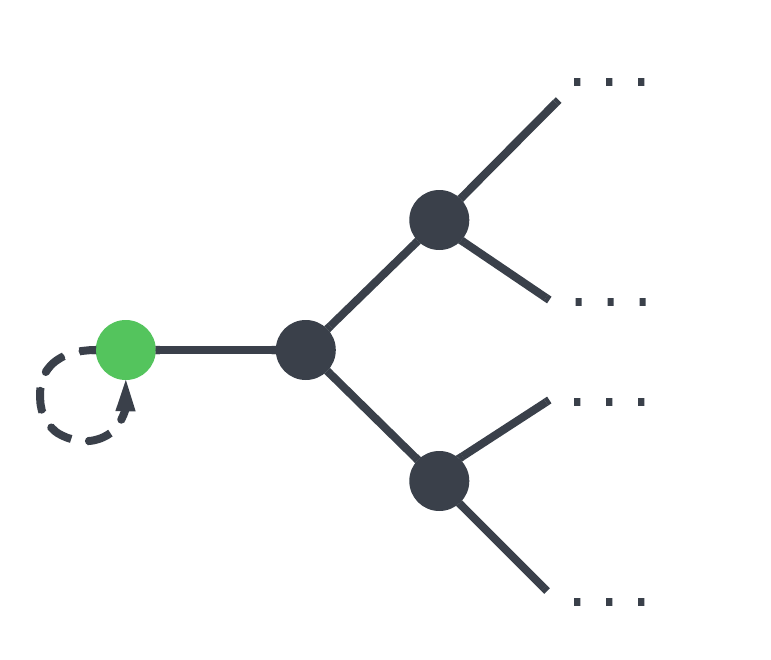}
\end{center}

Class 7 graphs might intuitively (but not precisely) be considered a special case of the soon-to-be-defined class 8 graphs; two entries in the base are equal.

\begin{center}
\includegraphics[width=\columnwidth]{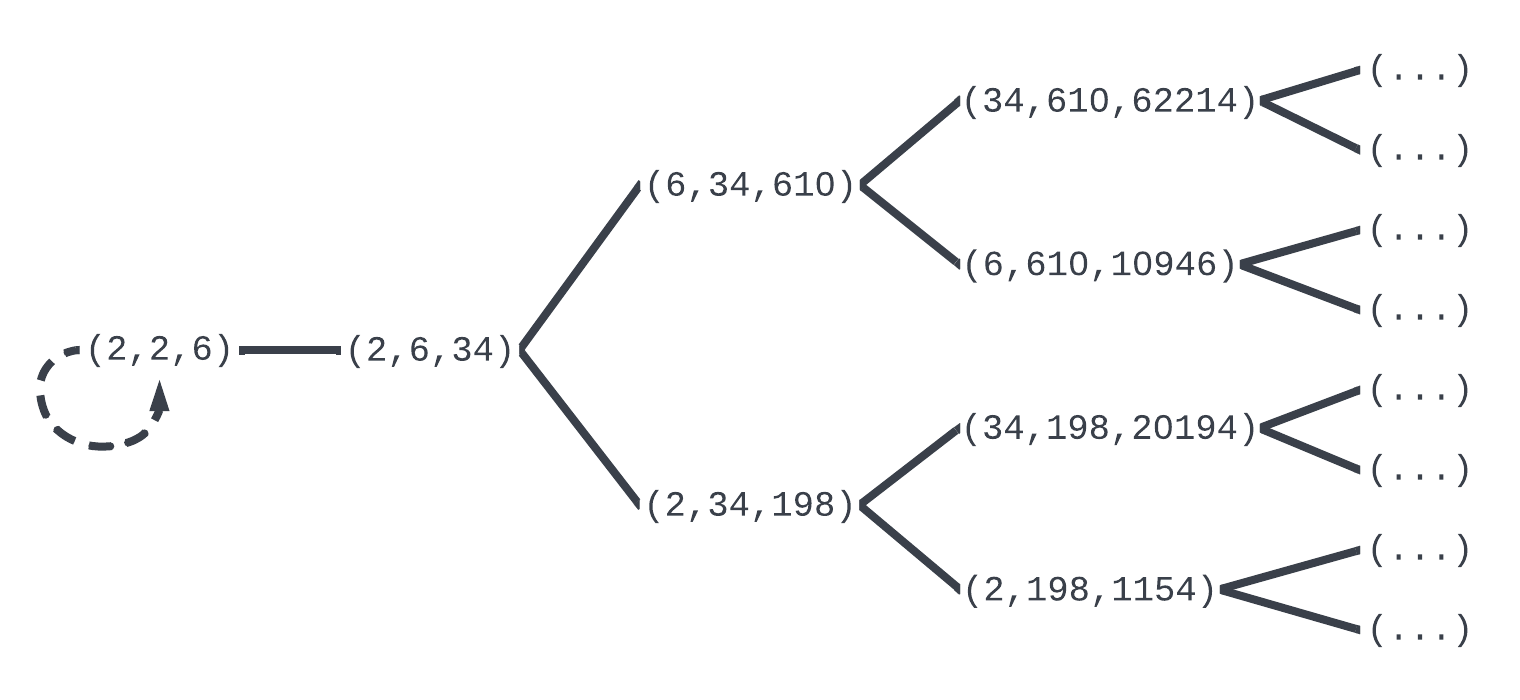}
\end{center}

An example is the graph with base $(2,2,6)$.

\subsection{Base $(n, 2m, 3nm)$}

Here, we use integers $n,m$ instead of $a,b$. This is because $n$ and $m$ need not be distinct, as opposed to $a,b$, which we state are distinct. Again, $n\neq 0$, $m\neq 0$. We restrict $n \not= 2m$, guaranteeing our graph does not belong to Class 7.

Clearly $ |n| < |3nm| $ and $ |2n| < |3nm| $, satisfying the requirements of Lemma \ref{binary}.
Furthermore, letting $c=3nm$,
$$ c' = 3\cdot n\cdot2m-c = 6nm - 3nm = 3nm = c. $$
We conclude that $ (n,2m,3nm) $ is the root of an infinite binary tree.

We call this graph {\bf Class 8}.

\begin{center}
\includegraphics[scale=0.25]{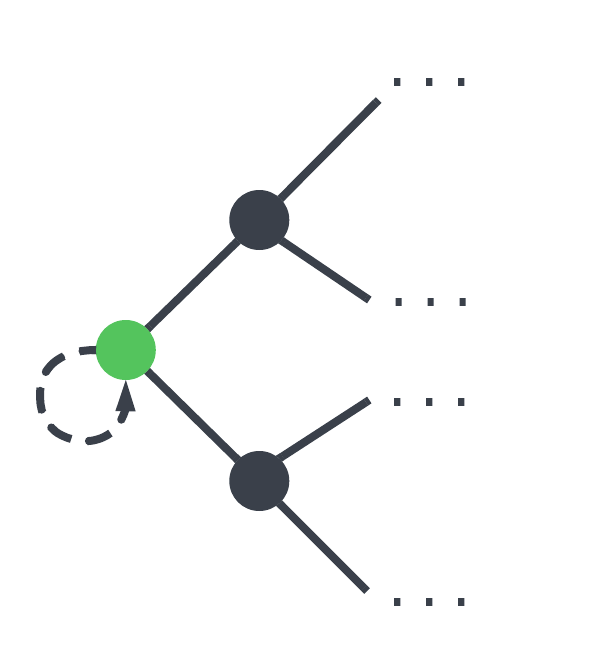}
\end{center}

An example is the graph with base $(1,2,3)$. Here, $n=m=1$.

\begin{center}
\includegraphics[width=\columnwidth]{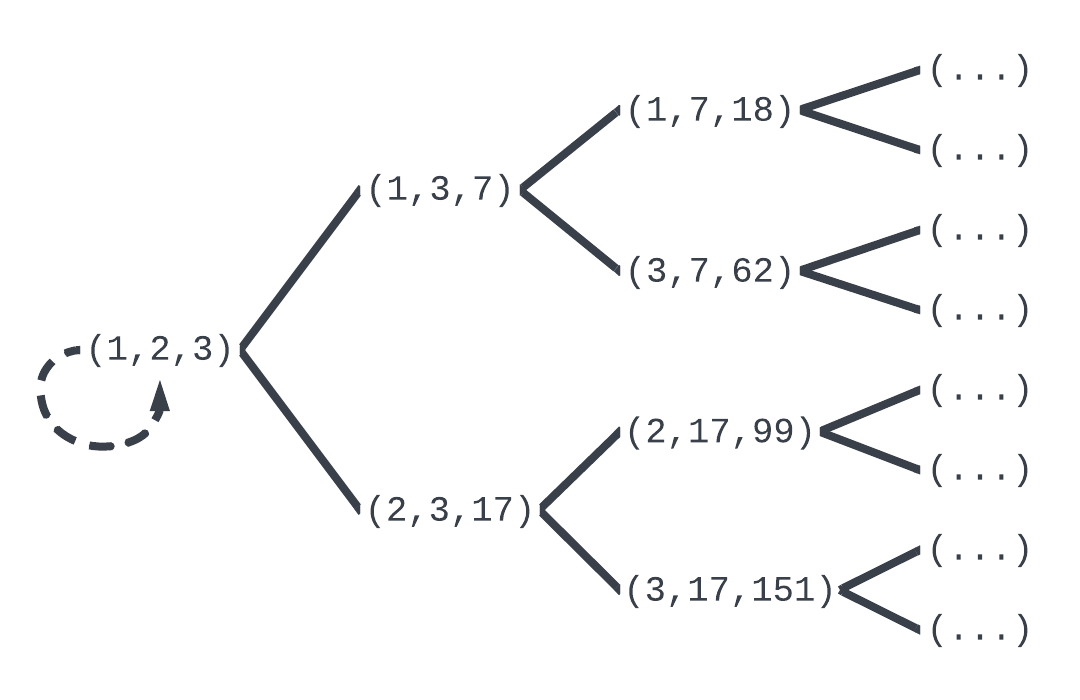}\newline
\end{center}

\subsection{Base $(a, b, c)$}
Here $ a \neq 0, b \neq 0, c \neq 0, a \neq b, a \neq c, b \neq c $, and we are not in Class 8.
Without loss of generality we assume $ |a| < |b| < |c| $.
In all cases we have 
$$ |a'| = |3bc-a| > |c| \; \mbox{and} \; |b'| > |3ac-b| > |c|, $$ 
which was established in Lemma 1.

However, it may be the case $ |c'| = |3ab-c| < |c| $, but in that case we know $ (a,b,c) $ is not a base.  
The precise condition for $ (a,b,c) $ to be a base is $ |3ab-c| \ge |c| $ 
Starting with $ (a,b,c) \leftrightarrow (a,b,c') $ where $ c' = 3ab-c $ we have three possibilities:
\begin{itemize}
    \item $ |c'| < |c|. $
    \item $ |c'| = |c|. $
    \item $ |c'| > |c|. $
\end{itemize}

When $ |c'| < |c| $ we know $ (a,b,c) $ is not a base.

When $ |c'| = |c| $ we have two cases:
\begin{itemize}
    \item $ 3ab-c = c. $
    \item $ 3ab-c = -c. $
\end{itemize}
If $ 3ab-c = c $, we know $ 3ab = 2c $ and either $ a $ or $ b $ are even.  
This would be a case of Class 8.
Here $ (a,b,c) $ has a link to itself, and is also the root of a binary tree.

If $ 3ab-c = -c $, we have either $ a=0 $ or $ b=0 $, which means we are in one of the earlier classes.

When $ |c'| > |c| $ we have $ (a',b,c) $, $ (a,b',c) $, and $ (a,b,c') $ each being the root of an infinite binary tree.  We call this graph {\bf Class 9}.

\begin{center}
\includegraphics[scale=0.25]{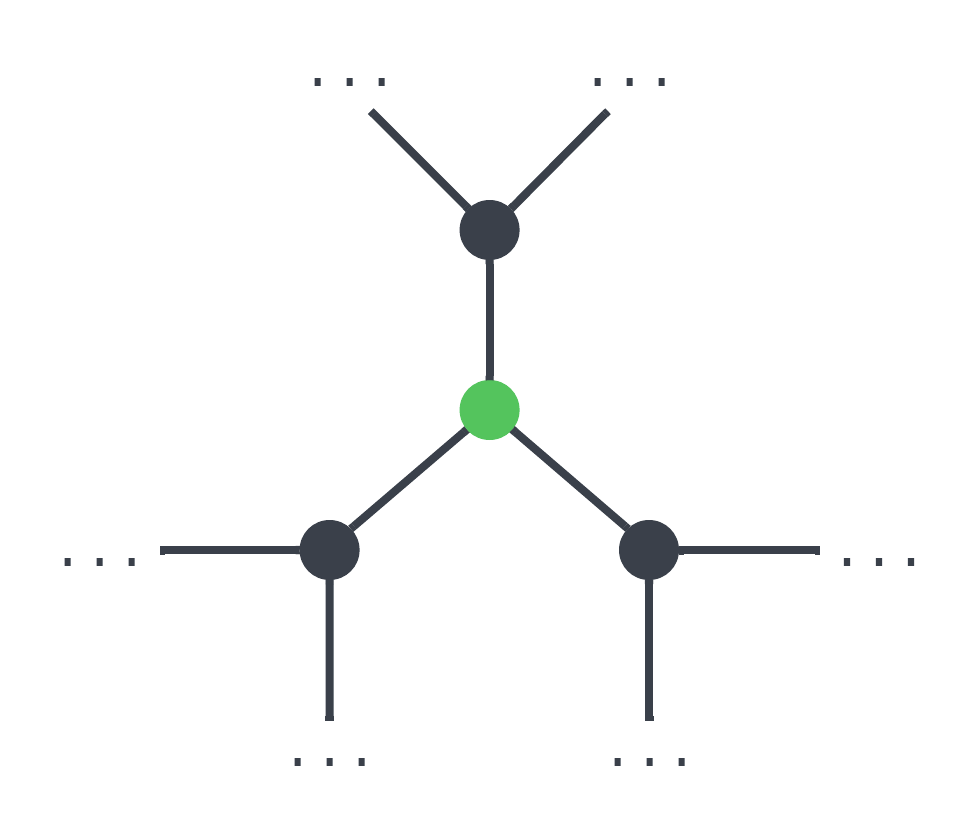}
\end{center}

An example is the graph with base $ (-12,1,3) $.

\begin{center}
\includegraphics[width=\columnwidth]{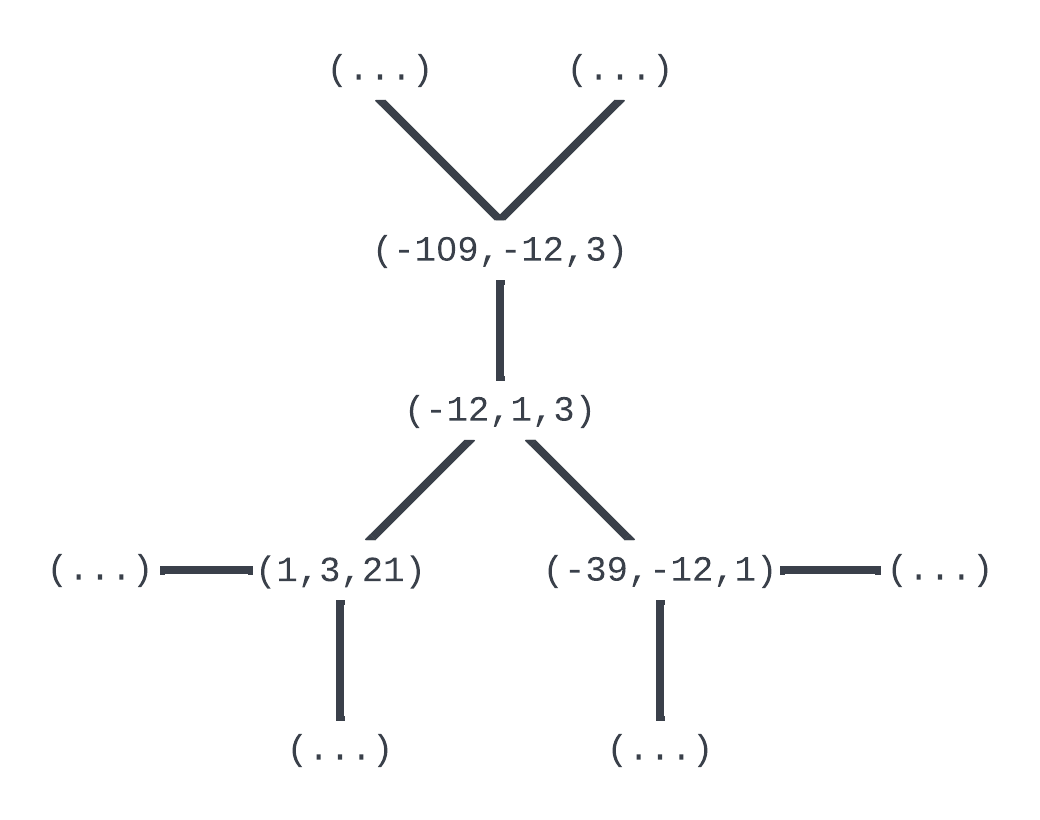}
\end{center}

\subsection{Distinctness of the classes}
To conclude our proof, we present the following argument demonstrating that no two classes are isomorphic. As stated above, this proof, and this argument, exclude loop edges. Were they included, a similar argument applies.

Classes 1 and 2 are the only finite graphs, and they have different numbers of vertices, hence are not isomorphic.

Class 3 is the only graph with a circuit.

Class 4 is the only graph with four vertices of degree 2.

Class 5 is the only graph with one vertex of degree 1, and one vertex of degree 2.

Class 6 is the only graph with exactly two vertices of degree 2.

Class 7 is the only graph with one vertex of degree 1, and no vertex of degree 2.

Class 8 is the only graph with only one vertex of degree 2, and no vertex of degree 1.

Class 9 is the only graph that is a 3-regular tree.

This completes the graph classification.  \qedsymbol

\section{Conclusion}
A reader may find curious the construction of graphs from triples of integers that are not actual solutions to the Markov equation.
Along this line of thought, we note that the Markov equation 
$$ x^2 + y^2 + z^2 = 3xyz $$
and the more general Diophantine equation
$$ x^2 + y^2 + z^2 = 3xyz + k, $$
share the same Vieta jumping relation.
For every integer triple $ (a,b,c) $ there exists a unique integer $ k $ for which the triple is a solution to the above equation.
Coupled with the Vieta jumping observation, this implies that every graph considered in this paper can be linked with a unique $ k $ value.
For example, the above graph with base $(-12, 1, 3)$ gives solutions for $ k=262 $.

It is known that for a given $ k $ value, there may be multiple corresponding graphs, for example consider $ k=81 $ and the graphs generated by the bases $(-1, 4, 4)$ and  $(0, 0, 9)$.
However, it is not known if there exist any values of $k$ with a unique corresponding graph. 
In particular, this is suspected for $k=1$.
We pose a question to build on the results of this paper.
\begin{question}
For which values of  $ k $, if any, are all solutions connected by a unique graph, or a unique class of graphs?
\end{question}

The scope of this investigation was limited to the integers. 
It is known that extending to the real numbers yields new isomorphism classes. 
For example, consider $(1/3, 1/3, 0)$, which generates a finite graph of 6 vertices. 
Extending further to the complex numbers, additional classes are expected, but not confirmed. 
\begin{question}
If triples of rational, real, and/or complex numbers are considered, what is the new extent of the set of graph isomorphism classes?
\end{question}

Other quadratic Diophantine equations may have other Vieta relations.
For a simple example, consider
$$ a^2 + b^2 = c^2 $$
with the relation
$$ (a,b,c) \leftrightarrow (a',b,c) \; \mbox{where} \; a' = -a .$$
\begin{question}
For other quadratic Diophantine equations with Vieta relations, what are the extent of the graph isomorphism classes?
\end{question}

\bibliographystyle{plainnat}
\bibliography{ref}

\end{document}